\documentclass{elsarticle}

\usepackage[english]{babel}
\usepackage{dsfont} 
\usepackage{graphicx}
\usepackage{color}
\usepackage[hypertexnames=false]{hyperref} 
\usepackage{enumerate}
\usepackage{amssymb,empheq} 
\usepackage{amsthm} 
\usepackage{thmtools} 
\usepackage{cleveref}
\usepackage[dvipsnames]{xcolor}
\usepackage{subcaption}

\allowdisplaybreaks
\newtheorem{thm}{Theorem}[section]
\newtheorem{lem}[thm]{Lemma}

\theoremstyle{definition}
\newtheorem{definition}[thm]{Definition}
\newtheorem{example}[thm]{Example}

\newenvironment{dfn}{\begin{definition}}{\end{definition}}

\newtheorem{exm}[thm]{Example}

\theoremstyle{remark}
\newtheorem{rem}[thm]{Remark}

\newcommand{\exmsymbol}{\hfill$\circ$}

\newcommand{\cset}{\mathds{C}}
\newcommand{\kset}{\mathds{K}}
\newcommand{\nset}{\mathds{N}}

\newcommand{\rset}{\mathds{R}}

\newcommand{\diff}{\mathrm{d}}

\newcommand{\pos}{\mathrm{Pos}}

\newcommand{\tr}{\mathrm{tr}}

\newcommand{\Herm}{\mathrm{Herm}}
\newcommand{\Bor}{\mathrm{Bor}}
\newcommand{\Pos}{\pos}
\newcommand{\skl}[2]{\left\langle #1, #2 \right\rangle}

\newcommand{\cA}{\mathcal{A}}

\newcommand{\cH}{\mathcal{H}}

\newcommand{\cK}{\mathcal{K}}
\newcommand{\cL}{\mathcal{L}}

\newcommand{\cT}{\mathcal{T}}
\newcommand{\cX}{\mathcal{X}}

\author{Lars-Luca Langer\footnote{Email: lars-luca.langer@uni-konstanz.de}$^{\text{,a}}$}
\address{$^\text{a}$Department of Mathematics and Statistics, University of Konstanz, Universit\"atsstra{\ss}e 10, D-78464 Konstanz, Germany}

\journal{arXiv}

\title{Operator $K$-Positivity Preserver}

\begin{document}

\begin{abstract}
	We characterize positivity preserving maps $T: B(\cH)_h \otimes \rset[x_1, \dots, x_n] \to B(\cH)_h \otimes \rset[x_1, \dots, x_n]$ on $\rset^n$ and on compact sets $K \subseteq \rset^n$. 
	This also characterizes local operator moment sequences and general operator moment sequences via positivity preserving maps.
\end{abstract}

\begin{keyword}
linear operators\sep positivity preserver\sep moments
\MSC[2020] Primary 44A60; Secondary 47A57, 47B38.
\end{keyword}

\maketitle

\section{Introduction}
\label{sec:intro}

Let $n \in \nset$, $K \subseteq \rset^n$ be a closed set, and denote by $\rset[x_1, \dots, x_n]$ the set of polynomials in $n$ variables with real coefficients.
Denote by $\Pos_\rset(K) := \{p \in \rset[x_1, \dots, x_n] \mid p(x) \geq 0 \text{ for }x \in K\}$ the set of positive polynomials on $K$.
A real sequence $(s_\alpha)_{\alpha \in \nset_0^n}$ is called a $K$-moment sequence, if there exists a measure $\mu: \Bor(K) \to \rset$ supported on $K$, such that 
\[
	s_\alpha = \int_K t^\alpha ~\diff\mu(t)
\]
for all $\alpha \in \nset_0^n$.
Every linear map $T: \rset[x_1, \dots, x_n] \to \rset[x_1, \dots, x_n]$ has the \emph{canonical representation} 
\[
	T = \sum_{\alpha \in \nset_0^n} \frac{1}{\alpha!} \cdot q_\alpha \cdot \partial^\alpha
\]
with polynomials $q_\alpha \in \rset[x_1, \dots, x_n]$ for $\alpha \in \nset_0^n$.
Such a linear map $T$ is called a $K$-positivity preserver, if $T\Pos_\rset(K) \subseteq \Pos_\rset(K)$.
Borcea's Theorem \cite[Thm.\ 3.1]{borcea11} asserts that $T$ is a $\rset^n$-positivity preserver, if and only if, $(q_\alpha(y))_{\alpha \in \nset_0^n}$ is a $\rset^n$-moment sequence for all $y \in \rset^n$.
In \cite[Thm.\ 3.5]{didio25KPosPresGen}, di Dio and Schmüdgen generalized this to any closed $K\subseteq\rset^n$, i.e., $T$ is a $K$-positivity preserver if and only if $(q_\alpha(y))_{\alpha\in\nset_0^n}$ is a $(K-y)$-moment sequence for all $y\in K$.
In \cite{didiolang26_matrixpospres} this connection between $K$-positivity preserving maps and general $K$-moment sequences was extended to matrix polynomials $p \in \Herm_m \otimes \rset[x_1, \dots, x_n]$ for $m \in \nset$.
In the present paper we extend these results further to operator polynomials $p \in B(\cH)_h \otimes \rset[x_1, \dots, x_n]$ for an arbitrary Hilbert space $\cH$.
We use operator-valued measures $\mu: \Bor(\rset^n) \to B(\cH)_h$ which are used in quantum physics for describing states of observables, see e.g.\ \cite{neumark40}, \cite{brandt99}, \cite{dubin14}, \cite{moretti17}. 
In quantum physics, the operators $\mu(S) \in B(\cH)_h$ for Borel sets $S \in \Bor(\rset^n)$ encode probability distributions of measurement outcomes. 
Therefore, characterizing operator moment sequences is of interest and has been done in \cite{curto25} for the univariate moment problem to a large degree. 
A Borcea type theorem for characterizing operator moment sequences and positivity preserving maps, however, has not yet been given and will be investigated in the current work. 
Additionally, contrary to \cite{curto25}, we will work in the general multivariate case.

This paper is structured in the following way. 
After the preliminaries, in \Cref{sec:BorceaCharacterizationOnK}, we give our main result (\Cref{thm:BorceaKClosedPolynomialOperators}) which characterizes positivity preservers on closed sets $K \subseteq \rset^n$. 
Afterwards we discuss the shortcomings of this theorem which stem from the difference of local and general operator moment problems.
In \Cref{sec:BorceaCharacterizationOnCompactSets} we remedy these problems by restricting ourself to compact sets $K \subseteq \rset^n$.

\section{Preliminaries}
\label{sec:prelim}

\subsection{Operator Polynomials}

Let $\cK \in \{\rset, \cset\}$ and let $\cH$ be a $\kset$-Hilbert space with inner product  $\skl{\,\cdot\,}{\,\cdot\,}_\cH$ and induced norm $||\,\cdot\,||_\cH$.
Let $B(\cH)_h$ be the Banach space of all linear, bounded, and self-adjoint operators on $\cH$.
Let $B(\cH)_{h, +}$ be the set of all positive operators in $B(\cH)_h$.	
Then $B(\cH)_h$ is a real vector space and $B(\cH)_{h, +} \subseteq B(\cH)_h$ is a full-dimensional convex cone with non-empty interior.
Let $n \in \nset$ and let $B(\cH)_h \otimes \rset[x_1, \dots, x_n]$ be the set of all \emph{operator polynomials}, i.e., linear maps $p: \rset^n \to B(\cH)_h$ of the form 
\[
	p = \sum_{\alpha \preceq \beta} p_\alpha x^\alpha
\]
for some $\beta \in \nset_0^n$, $p_\alpha \in B(\cH)_h$, see e.g.\ \cite{rodman89}.
An operator polynomial $p \in B(\cH)_h \otimes \rset[x_1, \dots, x_n]$ is said to be positive on $K \subseteq \rset^n$, i.e., $p \succeq 0$ on $K$, if $p(x) \in B(\cH)_{h, +}$ for all $x \in K$.
We define
\[\Pos(K) := \big\{ p \in B(\cH)_h \otimes \rset[x_1, \dots, x_n] \;\big|\; p(x) \succeq 0\ \text{ for all } x \in K \big\}.\]

\subsection{Operator-Valued Measures}

\begin{dfn}[see \protect{\cite[Def.\ 1 and Thm.\ 1]{berb09}}]
\label{dfn:OperatorValuedMeasure}
Let $\cH$ be a Hilbert space and $(\cX,\cA)$ be a measurable space. 
We say $\mu: \cA \to B(\cH)_h$ is a \emph{positive operator-valued measure} on $(\cX,\cA)$ if
\[\mu_x := \skl{\mu(\,\cdot\,)x}{x}_\cH: \cA \to \rset\]
is a finite positive measure on $(\cX, \cA)$ for all $x \in \cH$.
\end{dfn}

The definition of positive operator-valued measures coincides with the definitions in \cite[Rem.\ 2]{cimpzalar13}.
Note that only \emph{positive} operator-valued measures are defined which are always countably additive in the strong operator topology, see e.g.\ \cite[Prop.\ 1]{berb09}.
If one wishes to work with signed operator-valued measures, then one needs to pay attention to the convergence of countable additivity. 

For integration with respect to operator-valued measures, see \cite[Sec.\ 5]{berb09} or \cite[Rem.\ 4]{cimpzalar13}.
For every operator-valued measure $\mu$ on a measurable space $\cX$ and every integrable function $f$, one finds a linear bounded self-adjoint operator denoted by 
\[
	\int_\cX f(t) ~\diff\mu(t) \in B(\cH)_h
\]
such that the usual integral properties are satisfied, i.e., linearity in $f$ and positive functions leading to positive integrals, see \cite[Thm.\ 10]{berb09}.

In the matrix case, integration was defined using an inner product on $B(\cH)_h$ with the trace of matrices, see \cite{didiolang26_matrixpospres}. 
For infinite-dimensional Hilbert spaces, $B(\cH)_h$ is not a Hilbert space anymore and a trace does not exist for every operator.

\subsection{Operator Moment Sequences}

There are two different types of moment sequences for operators, general moment sequences and local moment sequences. 
For an overview about these two types in the univariate case see \cite{curto25}. 
We will work with the multivariate case instead.
\begin{definition}
Let $\cH$ be a Hilbert space, $n \in \nset$, $K \subseteq \rset^n$ be closed, and let $\cT = (T_k)_{k \in \nset_0} \subseteq B(\cH)_h$ be a sequence of operators. 
$\cT$ is called an \emph{operator $K$-moment sequence}, if there exists a positive operator-valued measure $\mu: \Bor(K) \to B(\cH)_h$ such that 
\[
	T_k = \int_K t^k ~\diff\mu(t)
\]
for all $k \in \nset_0$.
\end{definition}
\begin{definition}
Let $\cH$ be a Hilbert space, $n \in \nset$, $K \subseteq \rset^n$ be closed, and let $\cT = (T_k)_{k \in \nset_0} \subseteq B(\cH)_h$ be a sequence of operators. 
$\cT$ is called a \emph{local operator $K$-moment sequence}, if for every $x \in \cH$, $(\skl{T_kx}{x}_\cH)_{k \in \nset} \subseteq \rset$ is a $K$-moment sequence. 
That is, $\cT$ is a local operator $K$-moment sequence, if for every $x \in \cH$, there exists a measure $\mu_x: \Bor(K) \to \rset$ such that
\[
	\skl{T_kx}{x}_\cH = \int_K t^k ~\diff\mu_x(t)
\]
for all $k \in \nset_0$.
\end{definition}
On compact sets $K \subseteq \rset^n$, both definitions coincide, i.e., every local operator $K$-moment sequence is also an operator $K$-moment sequence, see \cite[Thm.\ 2]{bisgaard94} or \cite[Thm.\ 3.3]{curto25} and the discussion thereafter.

\subsection{Map-Valued Measures}

We will also use another type of measure called map-valued measures. 
\begin{definition}[see e.g.\ \protect{\cite[Sec.\ 2]{cimpzalar13}}]
Let $\cH, \cK$ be Hilbert spaces and $(\cX, \cA)$ be a measurable space. 
We say $\mu: \cA \to \cL(B(\cH)_h, B(\cK)_h)$ is a \emph{map-valued measure} on $(\cX, \cA)$, if 
\[
	\mu[A] := \mu(\,\cdot\,)(A): \cA \to B(\cK)_h
\]
	is an operator-valued measure for all $A \in B(\cH)_h$. 
	A map-valued measure is called \emph{positive} if $\mu[A]$ is a positive operator-valued measure for all $A \in B(\cH)_{h, +}$.
\end{definition}
Both operator- and map-valued measures were defined and used in \cite{cimpzalar13} but with different terminology. 
Both types of measures were denoted as operator-valued measures in \cite{cimpzalar13} and a positive map-valued measure was called a non-negative operator-valued measure. 
For the purpose of clarity, we will distinguish between these two types of measures as \emph{operator}-valued and \emph{map}-valued measures.
Integration with respect to map-valued measures is defined the same as in \cite[Sec.\ 2]{cimpzalar13}. 
Let $p \in B(\cH)_h \otimes \rset[x_1, \dots, x_n]$.
Then there exist $m \in \nset$, $A_1, \dots, A_m \in B(\cH)_h$, and $p_1, \dots, p_m \in \rset[x_1, \dots, x_n]$ with 
\[
	p = A_1 p_1 + \dots + A_m p_m.
\]
For a map-valued measure $\mu: \Bor(K) \to \cL(B(\cH)_h, B(\cK)_h)$ and $K \subseteq \rset^n$ closed, 
\[
	\int_K p ~\diff\mu := \int_K p_1 ~\diff\mu[A_1] + \dots + \int_K p_m ~\diff\mu[A_m].
\]
This integral does not depend on the representation $A_1p_1 + \dots + A_m p_m$. 
Additionally for a positive map-valued measure $\mu: \Bor(K) \to \cL(B(\cH)_h, B(\cK)_h)$ with $K \subseteq \rset^n$ closed, and a positive operator polynomial $p \in B(\cH)_h \otimes \rset[x_1, \dots, x_n]$, 
\[
	\int_K p ~\diff\mu \succeq 0.
\]

\subsection{Haviland Theorems}

\begin{thm}[see \protect{\cite[Thm.\ 3]{cimpzalar13}}]
\label{thm:HavilandTheoremOperators}
Let $n \in \nset$, $K \subseteq \rset^n$ be closed, and $\cH$ be a Hilbert space.
Let
\[L: B(\cH)_h \otimes \rset[x_1, \dots, x_n] \to \rset\]
be a linear functional.
Then the following are equivalent:
\begin{enumerate}[(i)]
\item $L$ preserves positivity on $K$, i.e., $L(p) \geq 0$ for all $p \in \pos(K)$.

\item $L$ is a $K$-moment functional, i.e., there exists a positive Borel map-valued measure $\mu: \Bor(K) \to \cL(B(\cH)_h, \rset)$ on $K$ such that 
\[
	L(p) = \int_K p(x) ~\diff\mu(x)
\]
for all $p \in B(\cH)_h \otimes \rset[x_1, \dots, x_n]$.
\end{enumerate}
\end{thm}

\begin{thm}[see \protect{\cite[Thm.\ 4 and 5]{cimpzalar13}}]
\label{thm:PositiveMapCompactSetExistsPositiveMeasure}
Let $n \in \nset$, $K \subseteq \rset^n$ be compact, and $\cH$ be a Hilbert space.
Let $L: B(\cH)_h \otimes \rset[x_1, \dots, x_n] \to B(\cH)_h$ be linear.
Then the following are equivalent:
\begin{enumerate}[(i)]

\item $L$ preserves positivity on $K$, i.e., $L(p) \succeq 0$ for all $p \in \pos(K)$.
\item There exists a unique positive Borel map-valued measure
\[\mu: \Bor(K) \to \cL(B(\cH)_h, B(\cH)_h)\]
on $K$ such that 
\[L(p) = \int_K p(x) ~\diff\mu(x)\]
for all $p \in B(\cH)_h \otimes \rset[x_1, \dots, x_n]$.
\end{enumerate}
\end{thm}

In \cite{cimpzalar13}, \Cref{thm:PositiveMapCompactSetExistsPositiveMeasure} was given with additional restrictions on the functional $L$. 
These are not needed when $K \subseteq \rset^n$ is compact.

\subsection{Canonical Representation of Linear Operators}

\begin{lem}[see \protect{\cite[Lem.\ 3.1]{didiolang26_matrixpospres}}]
\label{lem:DiffSumRepresentationRLinearOp}
Let $n \in \nset$ and let $V$ be a real vector space. 
Let
\[T: V \otimes \rset[x_1, \dots, x_n] \to V \otimes \rset[x_1, \dots, x_n]\]
be a linear operator, i.e., $T(af+g) = a\cdot Tf + Tg$ for all $a\in\rset$ and $f,g\in V \otimes \rset [x_1, \dots, x_n]$. 
Then, for all $\alpha \in \nset_0^n$, there exist unique linear operators
\[Q_\alpha: V \to V \otimes \rset[x_1, \dots, x_n]\]
such that
\begin{equation}\label{eq:DiffSumRepresentationQalpha}
T = \sum_{\alpha \in \nset_0^n} \frac{1}{\alpha!}\cdot Q_\alpha \times \partial^\alpha
\quad\text{with}\quad
(Q_\alpha \times \partial^\alpha)(v \otimes x^\beta) := Q_\alpha(v)\cdot\partial^\alpha x^\beta
\end{equation}
for all $\alpha, \beta \in \nset_0^n$ and $v \in V$.
\end{lem}

The representation (\ref{eq:DiffSumRepresentationQalpha}) is called \emph{canonical representation}.

\begin{rem}[see \protect{\cite[Rem.\ 3.2]{didiolang26_matrixpospres}}]
\label{rem:QalphaFormula}
The $Q_\alpha$ for $\alpha\in\nset_0^n$ from \Cref{lem:DiffSumRepresentationRLinearOp} are constructed recursively. 
Using the binomial transform, we get the explicit form
\[
Q_\beta(v) = \sum_{\alpha \preceq \beta} \binom{\beta}{\alpha}\cdot (-1)^{\beta-\alpha}\cdot T(v \otimes x^\alpha)\cdot x^{\beta-\alpha}
\]
for all $\beta \in \nset_0^n$ and $v \in V$.
\exmsymbol
\end{rem}

\begin{example}
	Let $n = 1$, $\cH = L^2([1, 3])$, and $y \in \rset$. 
	Define the operator  
\[
	T: B(\cH)_h \otimes \rset[x] \to B(\cH)_h \otimes \rset[x] 
	\quad \text{ by }\quad
	T(A \otimes x^k) = \widetilde{T}(A) \otimes (x+y)^k
\]
	for $A \in B(\cH)_h$, $k \in \nset$, and $x \in \rset$ with 
\[
	\widetilde{T}: B(\cH)_h \to B(\cH)_h, \quad
	\left[\widetilde{T}(A)f\right](t) = t \cdot Af(t)
\]
	for $A \in B(\cH)_h$, $f \in L^2([1, 3])$, $t \in [1, 3]$. 
	Define $Q_m: B(\cH)_h \to B(\cH)_h \otimes \rset[x]$,  
\begin{align*}
	Q_m(A)(x) 
	&= \sum_{k = 0}^m \binom{m}{k} (-1)^{m-k} \cdot T(A \otimes x^k) \cdot x^{m-k}
\\	&= \widetilde{T}(A) \otimes \left(\sum_{k = 0}^m \binom{m}{k} (-1)^{m-k} (x+y)^k x^{m-k}\right)
	= y^m \cdot \widetilde{T}(A)
\end{align*}
	for $m \in \nset$, $A \in B(\cH)_h$, and $x \in \rset$. 
	Then, by \Cref{rem:QalphaFormula}, 
\[
	T = \sum_{k = 0}^\infty \frac{1}{k!} Q_k \times \partial^k.
	\tag*{$\circ$}
\]
\end{example}

\section{Operator $K$-Positivity Preservers for closed $K \subseteq \rset^n$}
\label{sec:BorceaCharacterizationOnK}

Now we state the main theorem of this paper which is a generalization of \cite[Thm.\ 3.5]{didio25KPosPresGen} and \cite[Thm.\ 4.2]{didiolang26_matrixpospres}.
It characterizes $K$-positivity preservers of operator polynomials for closed $K \subseteq \rset^n$.

\begin{thm}
\label{thm:BorceaKClosedPolynomialOperators}
	Let $\cH$ be a Hilbert space, $n \in \nset$, $K \subseteq \rset^n$ be closed, and let $T: B(\cH)_h \otimes \rset[x_1, \dots, x_n] \to B(\cH)_h \otimes \rset[x_1, \dots, x_n]$ be a linear operator on operator polynomials with 
\begin{equation}\label{eq:DiffSumReprT}
	T = \sum_{\alpha \in \nset_0^n} \frac{1}{\alpha!}\cdot Q_\alpha \times \partial^\alpha
\end{equation}
	and linear operators $Q_\alpha: B(\cH)_h \to B(\cH)_h \otimes \rset[x_1, \dots, x_n]$, $\alpha \in \nset_0^n$.
	Then the following is equivalent:
\begin{enumerate}[(i)]
\item $T$ preserves positivity on $K$, i.e., $T\Pos(K) \subseteq \Pos(K)$
\item For every $y \in K, a \in \cH$, there exist positive Borel map-valued measures $\mu_{y, a}: \Bor(K) \to \cL(B(\cH)_h, \rset)$ such that 
\[
	\skl{\left[Tp(y)\right]a}{a}_\cH = \int_K p(t) ~\diff\mu_{y, a}(t)
\]
	for all $p \in B(\cH)_h \otimes \rset[x_1, \dots, x_n]$.
\item For every $y \in K, a \in \cH$, there exist positive map-valued measures $\nu_{y, a}: \Bor(K-y) \to \cL(B(\cH)_h, \rset)$ such that
\[
	\skl{Q_\alpha(A)(y)a}{a}_\cH = \int_{K-y} t^\alpha ~\diff\nu_{y, a}[A](t)
\]
	for all $\alpha \in \nset_0^n$, $A \in B(\cH)_h$.
\end{enumerate}
\end{thm}
\begin{proof}
	(i) $\Rightarrow$ (ii):
	Let $y \in K$ and $a \in \cH$. 
	Define the linear functional 
\[
	L_{y, a}: B(\cH)_h \otimes \rset[x_1, \dots, x_n] \to \rset, \quad
	p \mapsto \skl{\left[Tp(y)\right]a}{a}_\cH.
\]
	Then $L_{y, a}p \geq 0$ for all $p \in \Pos(K)$. 
	(ii) now follows from \Cref{thm:HavilandTheoremOperators}.

	(ii) $\Rightarrow$ (i):
	Let $y \in K$ and $a \in \cH$. 	
	By \Cref{thm:HavilandTheoremOperators}, 
\[
	L_{y, a}: B(\cH)_h \otimes \rset[x_1, \dots, x_n], \quad L_{y, a}p := \skl{\left[Tp(y)\right]a}{a}_\cH
\]
	is a positive functional, i.e., $L_{y, a}p \geq 0$ for all $p \in \Pos(K)$. 
	This holds for all $a \in \cH$, hence $Tp(y) \succeq 0$ for all $p \in \Pos(K)$, $y \in K$, i.e., $T$ preserves positivity on $K$.

	(ii) $\Rightarrow$ (iii):
	Let $A \in B(\cH)_h$, $y \in K$, and $a \in \cH$. 
	By \Cref{rem:QalphaFormula}, 
\[
	Q_\beta(A)(y) = \sum_{\alpha \preceq \beta} 
		\binom{\beta}{\alpha} \cdot (-1)^{\beta-\alpha} \cdot T(A \otimes t^\alpha)(y) \cdot y^{\beta-\alpha}
\]
	for $\beta \in \nset_0^n$. 
	Hence, 
\begin{align*}
	\skl{Q_\beta(A)(y)a}{a}_\cH 
	&= \sum_{\alpha \preceq \beta} 
		\binom{\beta}{\alpha} \cdot (-1)^{\beta-\alpha} \cdot y^{\beta-\alpha} \cdot \skl{T(A \otimes t^\alpha)(y)a}{a}_\cH
\\	&= \sum_{\alpha \preceq \beta} 
		\binom{\beta}{\alpha} \cdot (-1)^{\beta-\alpha} \cdot y^{\beta-\alpha} \cdot \int_K A \otimes t^\alpha ~\diff\mu_{y, a}(t)
\\	&= \int_K A \otimes (t-y)^\beta ~\diff\mu_{y, a}(t)
\\	&= \int_{K-y} A \otimes t^\beta ~\diff\nu_{y, a}(t)
	= \int_{K-y} t^\beta ~\diff\nu_{y, a}[A](t)
\end{align*}
	for all $\beta \in \nset_0^n$ where $\nu_{y, a} := \mu_{y, a}( \,\cdot\, + y): \Bor(K-y) \to \cL(B(\cH)_h, \rset)$ is a positive map-valued measure. 

	(iii) $\Rightarrow$ (ii):
	Let $A \in B(\cH)_h$, $y \in K$, and $a \in \cH$.
	By (\ref{eq:DiffSumReprT}) and (iii), 
\begin{align*}
	\skl{\left[T(A \otimes x^\beta)(y)\right]a}{a}_\cH
	&= \sum_{\alpha \preceq \beta} \binom{\beta}{\alpha} \skl{Q_\alpha(A)(y)a}{a}_\cH \cdot y^{\beta-\alpha}
\\	&= \sum_{\alpha \preceq \beta} \binom{\beta}{\alpha} y^{\beta-\alpha} \int_{K-y} t^\alpha ~\diff\nu_{y, a}[A](t)
\\	&= \int_{K-y} (t+y)^\beta ~\diff\nu_{y, a}[A](t)
	= \int_K A \otimes t^\beta ~\diff\mu_{y, a}(t)
\end{align*}
	where $\mu_{y, a} := \nu_{y, a}(\,\cdot\, - y): \Bor(K) \to \cL(B(\cH)_h, \rset)$ is a positive map-valued measure. 
	By linearity, 
\[
	\skl{\left[Tp(y)\right]a}{a}_\cH = \int_K p(t) ~\diff\mu_{y, a}(t)
\]
	for all $p \in B(\cH)_h \otimes \rset[x_1, \dots, x_n]$, $y \in K$, and $a \in \cH$.
\end{proof}

	Note that (iii) from \Cref{thm:BorceaKClosedPolynomialOperators} asserts that $\left(Q_\alpha(A)(y)\right)_{\alpha \in \nset_0^n}$ is a local operator $(K-y)$-moment sequence for all $y \in K$, $A \in B(\cH)_{h, +}$. 
	We cannot hope to find general operator moment sequences to represent $\left(Q_\alpha(A)(y)\right)_{\alpha \in \nset_0^n}$ because of the following example which is based on \cite[Thm.\ 1]{bisgaard94}, see also \cite[Exm.\ 2.8]{curto25}.
\begin{exm}
	Let $n = 1$ and $\cH = \cset^2$. 
	Then $B(\cH)_h$ is the set of Hermitian matrices $\Herm_2 \in \cset^{2 \times 2}$ which is a real Hilbert space with inner product $\skl{X}{Y}_{B(\cH)_h} := \tr(XY) \in \rset$ and orthonormal basis 
\begin{align*}
	H_{1, 1} &:= \begin{pmatrix}1&0\\0&0\end{pmatrix}, 
	&H_{1, 2} &:= \frac{1}{\sqrt{2}}\begin{pmatrix}0&1\\1&0\end{pmatrix}, 
\\	H_{2, 1} &:= \frac{1}{\sqrt{2}}\begin{pmatrix}0&i\\-i&0\end{pmatrix}, 
	&H_{2, 2} &:= \begin{pmatrix}0&0\\0&1\end{pmatrix}.
\end{align*}
	For $k \in \nset_0$, define $Q_k: \Herm_2 \to \Herm_2 \otimes \rset[x]$, 
\begin{gather*}
	Q_k(H_{1, 1}) = Q_k(H_{1, 2}) = Q_k(H_{2, 1}) = Q_k(H_{2, 2}), 
\\	Q_0(H_{1, 1}) := \begin{pmatrix}4 & 0\\0 & 1\end{pmatrix}, \quad
	Q_1(H_{1, 1}) := \begin{pmatrix}0 & 2\\2 & 0\end{pmatrix}, \quad
	Q_2(H_{1, 1}) := \begin{pmatrix}1 & 0\\0 & 4\end{pmatrix}, 
\\	Q_{2k-1}(H_{1, 1}) := 0 \text{ for } k \geq 2, \quad 
	Q_{2k}(H_{1, 1}) := \begin{pmatrix}a_k & 0\\0 & a_k\end{pmatrix} \text{ for } k \geq 2
\end{gather*}
	with $a_n = 2^{(n+2)!}$ for $n \geq 0$. 
	Then for all $M \in \Herm_2$ and $y \in \rset$, $(Q_k(M)(y))_{k \in \nset_0}$ is a local operator $\rset$-moment measure, see \cite[Thm.\ 1]{bisgaard94}. 
	Hence, by \Cref{thm:BorceaKClosedPolynomialOperators}, $T: \Herm_2 \otimes \rset[x] \to \Herm_2 \otimes \rset[x]$ with $T = \sum_{k = 0}^\infty \frac{1}{k!} Q_k \times \partial^k$ preservers positivity.
	However, by \cite[Prop.\ 2.1]{schmued87}, for $H_{1, 1} \in \Herm_{2, +}$ and $y \in \rset$, $(Q_k(H_{1, 1})(y))_{k \in \nset_0}$ is not an operator $\rset$-moment sequence, since 
\[
	\begin{pmatrix}Q_0(H_{1, 1}) & Q_1(H_{1, 1})\\Q_1(H_{1, 1}) & Q_2(H_{1, 1})\end{pmatrix} 
	= \begin{pmatrix}
		4 & 0 & 0 & 2
	\\	0 & 1 & 2 & 0
	\\	0 & 2 & 1 & 0
	\\	2 & 0 & 0 & 4
	\end{pmatrix}
\]
	is not positive semi-definite.
\exmsymbol
\end{exm}
	The proof of \Cref{thm:BorceaKClosedPolynomialOperators} cannot yield an operator moment sequence for $(Q_\alpha(A)(y))_{\alpha \in \nset_0^n}$. 
	The problem stems from the step (i) $\Rightarrow$ (ii) where the Haviland theorem for operator polynomials (\Cref{thm:HavilandTheoremOperators}) is used. 
	In order to construct a measure $\nu_y: \Bor(K) \to \cL(B(\cH)_h, B(\cH)_h)$ with moments $(Q_\alpha(A)(y))_{\alpha \in \nset_0^n}$, the measures from the Haviland theorem must fulfill certain conditions. 
	These conditions are, for example, outlined in \cite[Thm.\ 2]{berb09} and ensure that one can construct a bounded sesquilinear form from which the operator $\nu_y(A)$ follows. 

	However, on compact sets $K \subseteq \rset^n$, the moment problem is determinate and every local operator moment sequence is a general operator moment sequence.

\section{Operator $K$-Positivity Preserver for Compact $K$}
\label{sec:BorceaCharacterizationOnCompactSets}

\begin{thm}
\label{thm:BorceaKCompactPolynomialOperators}
Let $n \in \nset$, $K \subseteq \rset^n$ be compact, and $\cH$ be a Hilbert space.
Let
\[
	T: B(\cH)_h \otimes \rset[x_1, \dots, x_n] \to B(\cH)_h \otimes \rset[x_1, \dots , x_n], 
	\ 
	T = \sum_{\alpha \in \nset_0^n} \frac{1}{\alpha!} \cdot Q_\alpha \times \partial^\alpha
\]
with linear operators $Q_\alpha: B(\cH)_h \to B(\cH)_h \otimes \rset[x_1, \dots, x_n]$, $\alpha \in \nset_0^n$.
Then the following are equivalent:
\begin{enumerate}[(i)]
\item $T$ preserves positivity on $K$, i.e., $T\pos(K) \subseteq \pos(K)$.

\item For each $y \in K$, there exists a positive map-valued measure $\nu_y: {\Bor(K-y)} \to \cL(B(\cH)_h, B(\cH)_h)$ such that 
\[
	Q_\beta(A)(y) = \int_{K-y} t^\beta ~\diff\nu_y[A](t)
\]
for all $\beta \in \nset_0^n$.
\end{enumerate}
\end{thm}

The proof is an adapted version of the proof of \cite[Thm.\ 3.5]{didio25KPosPresGen} using \cite[Thm.\ 4]{cimpzalar13} on compact sets and a generalization of \cite[Thm.\ 5.4]{didiolang26_matrixpospres}.

\begin{proof}
(i) $\Rightarrow$ (ii): 
Let $y \in K$ and define
\[L_y: B(\cH)_h \otimes \rset[x_1, \dots, x_n] \to B(\cH)_h,\quad p\mapsto L_y(p) := Tp(y).\]
Then $L_y(p) \succeq 0$ for all $p \in \Pos(K)$.
\Cref{thm:PositiveMapCompactSetExistsPositiveMeasure} asserts the existence of a positive Borel map-valued measure $\mu_y: \Bor(K) \to \cL(B(\cH)_h, B(\cH)_h)$ such that
\[L_y(p) = \int_K p(t) ~\diff \mu_y(t)\]
for all $p \in B(\cH)_h \otimes \rset[x_1, \dots, x_n]$. 
	By \Cref{rem:QalphaFormula}, 
\[
	Q_\beta(A)(y) = \sum_{\alpha \preceq \beta} 
		\binom{\beta}{\alpha} \cdot (-1)^{\beta-\alpha} \cdot T(A \otimes t^\alpha)(y) \cdot y^{\beta-\alpha}
\]
	for all $\beta \in \nset_0^n$. 
	Hence, 
\begin{align*}
	Q_\beta(A)(y)
	&= \sum_{\alpha \preceq \beta} 
		\binom{\beta}{\alpha} \cdot (-1)^{\beta-\alpha} \cdot y^{\beta-\alpha} \cdot T(A \otimes t^\alpha)(y)
\\	&= \sum_{\alpha \preceq \beta} 
		\binom{\beta}{\alpha} \cdot (-1)^{\beta-\alpha} \cdot y^{\beta-\alpha} \cdot \int_K A \otimes t^\alpha ~\diff\mu_y(t)
\\	&= \int_K A \otimes (t-y)^\beta ~\diff\mu_y(t)
	= \int_K (t-y)^\beta ~\diff\mu_y[A](t)
	= \int_{K-y} t^\beta ~\diff\nu_y[A](t)
\end{align*}
	for $\beta \in \nset_0^n$, $A \in B(\cH)_h$ where $\nu_y := \mu_y(\,\cdot\, + y): \Bor(K-y) \to \cL(B(\cH)_h, B(\cH)_h)$ is a positive map-valued measure.

(ii) $\Rightarrow$ (i):
Let $y \in K$ and let $\nu_y: \Bor(K-y) \to \cL(B(\cH)_h, B(\cH)_h)$ be a positive map-valued measure such that
\[
	Q_\beta(A)(y) = \int_{K-y} t^\beta ~\diff\nu_y[A](t)
\]
for all $A \in B(\cH)_h$ and all $\beta\in\nset_0^n$.
Define $\mu_y := \nu_y(\,\cdot\, - y): \Bor(K) \to \cL(B(\cH)_h, B(\cH)_h)$ as the pushforward measure.
Then
\begin{align*}
\int_K A \otimes t^\beta ~\diff\mu_y(t)
	&= \int_{K-y} A \otimes (t+y)^\beta ~\diff\nu_y(t)
\\	&= \sum_{\alpha \preceq \beta} \binom{\beta}{\alpha}\cdot y^{\beta-\alpha}\cdot \int_{K-y} t^\alpha ~\diff\nu_y[A](t)
\\	&= \sum_{\alpha \in \nset_0^n} \frac{1}{\alpha!}\cdot Q_\alpha(A)(y)\cdot \partial^\alpha y^\beta
	= T(A \otimes x^\beta)(y)
\end{align*}
for all $\beta \in \nset_0^n$, $A \in B(\cH)_h$, and $y \in K$.
By linearity,
\[
	\int_K p(t) ~\diff\mu_y(t) = Tp(y)
\]
for all $p \in B(\cH)_h \otimes \rset[x_1, \dots, x_n]$ and $y \in K$.
Hence, by \Cref{thm:PositiveMapCompactSetExistsPositiveMeasure}, $T$ is a $K$-positivity preserver.
\end{proof}

The proof of (ii) $\Rightarrow$ (i) from \Cref{thm:BorceaKCompactPolynomialOperators} also states, that if one of the equivalent conditions from \Cref{thm:BorceaKCompactPolynomialOperators} is fulfilled, then there exist positive map-valued measures $\mu_y: \Bor(K) \to \cL(B(\cH)_h, B(\cH)_h)$ for $y \in K$ such that 
\[
	Tp(y) = \int_K p(t) ~\diff\mu_y(t)
\]
	for all $p \in B(\cH)_h \otimes \rset[x_1, \dots, x_n]$ and $y \in K$. 
	Additionally, (ii) states that for every $A \in B(\cH)_{h, +}$ and $y \in K$, the sequence $(Q_\alpha(A)(y))_{\alpha \in \nset_0^n}$ is an operator $(K-y)$-moment sequence.

\section*{Acknowledgments}
The author thanks Philipp di Dio for his advise on the paper.

\section*{Funding}

The author is supported by the Deutsche Forschungs\-gemein\-schaft DFG with the grant DI-2780/2-1 of Philipp di Dio.

\bibliographystyle{amsalpha}

\providecommand{\bysame}{\leavevmode\hbox to3em{\hrulefill}\thinspace}
\providecommand{\MR}{\relax\ifhmode\unskip\space\fi MR }
\providecommand{\MRhref}[2]{%
  \href{http://www.ams.org/mathscinet-getitem?mr=#1}{#2}
}
\providecommand{\href}[2]{#2}

\end{document}